\RequirePackage{fix-cm}
\documentclass[smallextended,envcountsect]{svjour3}       
\smartqed  
\usepackage{tikz}
\colorlet{mygreen}{green!60!black}
\newcommand{\mygreen}[1]{{\color{mygreen}#1}}
\usepackage{graphicx}
\usepackage{bigints}
\usepackage{amssymb}
\usepackage{verbatim}
\usepackage{subfigure}
\usepackage{booktabs}
\usepackage{booktabs}
\usepackage{amsfonts,amsmath}
\usepackage{capt-of}
\usepackage{hyperref}
\hypersetup{
	colorlinks=true,
	linkcolor=blue,
	filecolor=blue,      
	urlcolor=blue,
	citecolor=blue
}
\usepackage{cite}
\newtheorem{assumption}{Assumption}[section]

\begin{document}
	
	\title{Regularized neural network for general variational inequalities involving monotone couples of operators in Hilbert spaces}
	
	\titlerunning{Regularized neural networks for monotone GVIs}        
	
	\author{Pham Ky Anh         \and
		Trinh Ngoc Hai\thanks{Dedicated to Professor Hoang Xuan Phu on the occasion of his 70th birthday.}  \and 
            Nguyen Van Manh
        	}
	
	
	\institute{Pham Ky Anh \at
		Department of Mathematics, Vietnam National University, 334 Nguyen Trai, Thanh Xuan, Hanoi, Vietnam\\
		\email{anhpk\symbol{64}vnu.edu.vn}           
		\and
		Trinh Ngoc Hai \at
		School of Applied Mathematics and Informatics,
		Hanoi University of Science and Technology, 1 Dai Co Viet, Hai Ba Trung, Hanoi, Vietnam\\
		\email{hai.trinhngoc\symbol{64}hust.edu.vn}
        \and
		Nguyen Van Manh \at
            Faculty of Basic Sciences, Hanoi University of Industry, 298 Cau Dien, Bac Tu Liem, 100000, Hanoi, Vietnam\\
		\email{nvmanhhhn\symbol{64}haui.edu.vn}
	}
	
	\date{Received: date / Accepted: date}

	\maketitle
	
	\begin{abstract} 
	In this paper, based on the Tikhonov regularization technique, we study a monotone general variational inequality (GVI)  by considering an associated strongly monotone GVI, depending on a regularization parameter $\alpha,$ such that the latter admits a unique solution $x_\alpha$ which tends to some solution of the initial GVI, as $\alpha \to 0.$ However, instead of solving  the regularized GVI for each $\alpha$,  which may be very expensive, we consider a neural network (also known as a dynamical system) associated with the regularized GVI and establish the existence and the uniqueness of the strong global solution to the corresponding Cauchy problem. An explicit discretization of  this neural network leads to strongly convergent iterative regularization algorithms for monotone general variational inequality. Numerical tests are performed to show the effectiveness of the proposed methods.\\
    This work extends our recent results in  [Anh, Hai, Optim. Eng. 25 (2024) 2295-2313] to more general setting.
		\keywords{General variational inequality \and Bilevel variational inequality \and Regularized neural network \and Monotonicity \and Monotone couple \and Lipschitz continuity}
		\subclass{47J20 \and 49J40 \and 49M30}  
	\end{abstract}
	\section{Introduction and Preliminaries}
Let $H$ be a real Hilbert space, whose inner product and norm are
denoted by $\langle \cdot, \cdot \rangle $ and $\| \cdot\| $, respectively. Let $C \subset H$ be a nonempty, closed and convex set and $ A; B; F; G : H \to H$
be given operators. \\
The  variational  inequality of $G$ on $C$ is
	\begin{equation}\label{VI}
		\text{to find } x^* \in C \text{ such that }   \left\langle G x^*, y-  x^*  \right\rangle \geq 0 \quad \forall y \in C. \tag{VI($G,C$)}
	\end{equation}
The inverse variational  inequality of $B$ on $C$ is
	\begin{equation}\label{IVI}
		\text{to find } x^* \in H \text{ such that } Bx^* \in C \text{ and } \left\langle  x^*, y- B x^*  \right\rangle \geq 0 \quad \forall y \in C. \tag{IVI($B,C$)}
	\end{equation}
In this paper, we focus our attention on solving the following General Variational Inequality  (GVI, for short):
\begin{equation}\label{GVI}
{\rm Find}\; x^* \in H \;  {\rm such \;  that} \; Fx^* \in C \; {\rm and } \;\langle Ax^*, y - Fx^*\rangle  \geq 0 \quad  \; \forall y \in C.\;\; \tag{GVI($A, F, C$)}    
\end{equation}
Denote by Sol($A, F, C$) the solution set of GVI($A, F, C$), which is assumed not empty
throughout this paper.\\
A couple $(A, F)$ is said to be, see, \cite{LiZou},
\begin{itemize}
\item[] $\gamma$-strongly monotone, $\gamma >0$, if 
$$ \langle Ax - Ay, Fx - Fy \rangle \geq \gamma \|x - y\|^2, \; \forall x, y \in H; $$
\item[] monotone, if 
$$ \langle Ax - Ay, Fx - Fy \rangle \geq 0, \; \forall x, y \in H. $$
\end{itemize}
A general variational inequality GVI($A, F, C$), involving a monotone (strongly monotone) couple ($A, F$), is called a monotone (strongly monotone) GVI.\\
\indent In recent years there has been a growing interest in the general variational inequality driven by its large number of applications. When $F=I$-the identity operator, the GVI collapses into a variational inequality VI($A,C$) and if $A=I,$ we get the inverse variational inequality IVI($F, C$). Thus, the general variational inequality provides a unified formulation for several important problems. In what follows, only the directly involved literature on GVIs will be cited here. The general variational inequality was first introduced and studied by Noor \cite{Noor1}. Pang and Yao \cite{PangYao} established some sufficient conditions for the existence of the solutions to a GVI and investigated their stability properties. Zhao \cite{Zhao} and He \cite{He1} proposed some implicit methods for monotone general variational inequalities in finite dimensional Hilbert spaces. Noor et al \cite{AlShejari}, \cite{Noor1}, and Tangkhawiwetkul \cite{Tang} proposed some neural networks as well as their discretized versions for solving strongly monotone GVI. Based on the Banach contraction principle, they established strong convergence of iterative methods under very restrictive conditions. \\
In this paper, we combine the Tikhonv regularization technique and the dynamical system approach to construct a class of iterative regularization methods for solving monotone general variational inequalities. This distinguishes our methods from most other projection-type methods for general variational inequalities with monotone/strongly monotone couples of operators.\\
The paper is organized as follows. In Section 1, we briefly state some definitions and lemmas which will be subsequently used. In Section 2, the existence and uniqueness of the solution to strongly monotone GVIs is established under some mild conditions. In Section 3, a regularized neural network associated with a unique solution to the strongly monotone GVI depending on the regularization parameter $\alpha,$ is studied. It is proved that the corresponding Cauchy problem admits a unique global solution $x(t),$ whose limit at infinity exists and solves the given monotone GVI.  
In Section 4, we derive a class of explicit iterative regularization  methods by discretizing the mentioned above Cauchy problem. Finally in Section 5, several examples, including an infinite dimensional one, are given to demonstrate the applicability of our result.\\
Now, recall some notions for subsequent needs.\\
An operator $A:C\to H$ is said to be
\begin{enumerate}
			\item $\lambda$-strongly monotone on $C$, if there exists a constant $\lambda >0$ such that for all $x,y\in C$, we have
			$$ \langle Ax-Ay, x-y \rangle \geq   \lambda \|x-y\|^2.$$
			\item monotone on $C$, if  for all $x,y\in C$, it  holds that
			$$ \langle Ax - Ay,x-y \rangle \geq 0.$$
            \item firmly nonexpansive, if for all $x,y\in C$, one gets
            $$ \langle Ax - Ay,x-y \rangle \geq \| Ax - Ay \|^2 $$
			\item $L_A$-Lipschitz continuous on  $C$, if  for all $x,y\in C$, it  holds that
			$$ \| Ax-Ay\| \leq L_A \| x-y\|.   $$
			\item sequentially weak-to-weak continuous on $C$ if for any sequence $\{x^k\}$ satisfying $x^k \rightharpoonup \bar x$, it holds that $Ax^k \rightharpoonup A\bar x$.
\end{enumerate}
It is well known that the metric projection of $x\in H$ onto $C,$ defined by $Px:= \underset{y \in C} {\rm argmin}\| x- y\|$ enjoys two important properties
\begin{itemize}
\item[i)] $z=Px$ if and only if 
		$$
		\left\langle x-z,y-z \right\rangle   \leq 0 \quad \forall y \in C;
		$$ 
\item[ii)] $\langle Px - Py, x-y \rangle \geq \|Px - Py\|^2, \; \forall x, y \in H.$
\end{itemize}
The second property means that $P$ is firmly nonexpansive, hence, it is monotone and nonexpansive.
\vskip0.1cm
\indent We start with the following generalization of \cite[Lemma 2.2]{AnhHai}.
	\begin{lemma}\label{lemma1}
		\begin{itemize}
			\item[\textup{(a)}] Assume that the couple $(A,F)$ is $\gamma$-strongly monotone and $x^*$ is a solution of \ref{GVI}. Then, for all $f\in C, y\in H$, we have
			\begin{equation}\label{3a}
				\left\langle Fy-  f, Ay-Ax^*  \right\rangle +\left\langle  f-Fx^*, Ay  \right\rangle  \geq \gamma \| y-x^*  \|^2 .
			\end{equation}	
			\item[\textup{(b)}] Assume that $(A,F)$ is monotone. Then, $x^*$ is a solution of \ref{GVI} if and only if 
			\begin{equation}\label{1a}
				Fx^* \in C \textup{ and } \left \langle  Ay-Ax^*, Fy -f  \right \rangle + \left \langle  Ay, f-Fx^*  \right \rangle \geq 0 \quad \forall y\in H,  f\in C.
			\end{equation}		
		\end{itemize}
	\end{lemma}
	\begin{proof}
		(a)  Since $(A,F)$ is $\gamma$-strongly monotone, for all $f\in C, y\in H$ we have
		\begin{equation}\label{2}
			\left\langle Fy-  f, Ay-Ax^*  \right\rangle +\left\langle  f-Fx^*, Ay  \right\rangle  \geq \gamma \| y-x^*  \|^2 + \left\langle  f-Fx^*, Ax^*  \right\rangle.
		\end{equation}
		Since $x^*$ is a solution of \ref{GVI}, it holds that $\left\langle  f-Fx^*, Ax^*  \right\rangle \geq 0$. Combining this and \eqref{2}, we get the desired result.\\
		(b) Since $(A,F)$ is monotone, for all $f\in C, y\in H$ it holds that
		\begin{equation*}
			\left\langle Fy-  f, Ay-Ax^*  \right\rangle +\left\langle  f-Fx^*, Ay  \right\rangle  \geq \left\langle  f-Fx^*, Ax^*  \right\rangle \geq 0.
		\end{equation*}
		The necessity is established. Conversely, choosing $y=x^*$  in \eqref{1a}, we get the sufficiency. 
	\end{proof}
The following result was established in a finite dimensional space by He and Dong \cite{HeDong}, however it remains true in any real Hilbert spaces.
	\begin{lemma}\label{existuniq}\cite[Theorem 3.2]{HeDong}
		Let $C$ be a nonempty closed convex subset of a real Hilbert space $\mathcal{H},$ and
		let $B :\mathcal{H} \to \mathcal{H}$ be a Lipschitz continuous and strongly monotone operator. Then the inverse variational inequality IVI($B, C$) has a unique solution.
	\end{lemma}
	\begin{lemma}\label{lemma31}\cite{Xu}
		Let $\{U_k \}; \{\theta_k \}; \{\zeta_k \} $ be positive sequences satisfying $\{\theta_k\} \subset (0;1)$; $\sum_{k=1}^\infty \theta_k =\infty$,   $\lim\limits_{k \to \infty} \dfrac{\zeta_k}{\theta_k}=0 $ and
		$$
		U_{k+1} \leq (1-\theta_k) U_k  +\zeta_k\quad \forall k\geq 0.
		$$
		Then, $U_k \to 0$ as $k\to \infty$.
	\end{lemma}
\section{Strongly monotone general variational inequalities}
\begin{theorem}\label{Th2.1}
Suppose that the operators $A$ and $ F,$ acting in a finite dimensional Hilbert space $H,$ are Lipschiz continuous with coefficients $L_A$ and $L_F$, respectively.  If in addition, the couple ($A, F$) is $\gamma$-strongly monotone.  Then 
\begin{itemize}
\item[i.] $A, F$ admit Lipschitz continuous inverses, moreover, for all $u,   v \in H,$ 
\begin{equation}\label{1}
\frac{1}{L_F}\|u-v\| \leq \|F^{-1}u - F^{-1}v\| \leq \frac{L_A}{\gamma}\|u-v\|, 
\end{equation}
$$\frac{1}{L_A}\|u-v\| \leq \|A^{-1}u - A^{-1}v\| \leq \frac{L_F}{\gamma}\|u-v\|.  $$
\item[ii.]  The general variational inequality GVI($A, F, C$) admits a unique solution. Moreover, $x^*$ is a solution to the general variational inequality GVI($A, F, C$) if and only if $u^*:= F x^*$ is  a solution to the strongly monotone variational inequality VI($AF^{-1}, C$) and if and only if $v^*:= Ax^*$ is a solution to the inverse variational inequality IVI($FA^{-1}, C$).
\end{itemize}
\end{theorem}
\begin{proof}
From the $\gamma$-strong monotonicity of the couple $(A, F),$ one has
$$\gamma \|x - y\|^2 \leq \langle Ax -Ay, Fx-Fy \rangle \leq \|Ax-Ay\| \|Fx - Fy\|, \; \forall x, y \in H.$$ 
Using the last inequality and the Lipschitz continuity of $A,$ and $F,$ one finds\\
$$\|Fx-Fy\| \geq \frac{\gamma}{L_A}\|x -y\| \;\; {\rm and} \;\; \|Ax-Ay\| \geq \frac{\gamma}{L_F}\|x-y\|,$$
which show that $A, $ and $F$ are injective mappings. By Brouwer  invariance of domain theorem \cite{Brou}, the sets $A(H)$ and $F(H)$ are open. To prove that $A(H)$ is also closed, we assume $A(H) \ni y_n \to \tilde{y}.$ Choose $x_n \in H,$ s.t., $Ax_n = y_n.$ Since $\frac{\gamma}{L_F}\|x_n-x_m \| \leq \|Ax_n - Ax_m\| = \|y_n - y_m\|,$ the sequence $\{x_n\}$ is fundamental, hence $x_n \to \tilde{x}.$ Further, $\tilde{y} = \lim_{n\to \infty} y_n = \lim_{n\to \infty} Ax_n = A\tilde{x}.$ The last relation shows that $A(H)$ is closed, hence $A(H) =H.$ By the same argument, one can conclude that $F(H) = H.$ Thus both $A$ and $F$ are invertible, moreover the assertion \eqref{1} holds for $A^{-1}$ and $F^{-1}.$\\
Now suppose, $x^* \in {\rm Sol}(A, F, C)$ then $Fx^* \in C$ and $\langle Ax^*, y - Fx^*\rangle \geq 0 $ for all  $  y \in C.$ Denote $v^*:= A x^*, \; u^*:= F x^*,$ one gets $\langle v^*, y - FA^{-1}v^*\rangle \geq 0, \; FA^{-1}v^* \in C $ and $\langle AF^{-1} u^*, y - u^* \rangle \geq 0 $ for all $y \in C.$ That means $u^*$ is the unique solution to strongly monotone variational inequality VI($AF^{-1}, C$) and $v^*$ is the unique solution to the strongly inverse variational inequality IVI($FA^{-1}, C$). Conversely, if $u^* (v^*)$ is the unique solution to strongly monotone variational inequality VI($AF^{-1}, C$) (strongly monotone inverse variational inequality IVI($FA^{-1}, C$)), then $x^* = F^{-1}u^* = A^{-1}v^*$ is a unique solution to the GVI($A, F, C$). $\blacksquare$
\end{proof}
{\bf Remark 2.1}  The finite dimensionality of $H$ is used only for the proof of the invertibility of $A,$ and $F.$ The conclusions of Proposition 1 may not true, if $H$ is an infinite dimensional Hilbert space. In particular, the operators $A,  F$ may not invertible, even if ($A, F$) is a strongly monotone couple of  Lipschitz continuous operators. Indeed, let $H$ be a real seperable Hilbert space with an orthonormal basis $\{e_i\}_{i=1}^\infty$ and $Ax = Fx = \sum_{i=1}^\infty x_ie_{i+1},$  for any $x = \sum_{i=1}^\infty x_i e_i.$  Clearly, $A$ and $F$ are  bounded linear non-surjective operators with im $A$ = im $F$ = span$(\{e_i\}_{i=2}^\infty) \ne H$. However, the couple ($A,F$) is $1$-strongly monotone.
\vskip0.1cm
\begin{theorem}\label{Th2.2}
Let $C$ be a nonempty closed convex subset in a real Hilbert space $H,$ and $A, F: H \to H$ be Lipschitz continuous operators on $H$ with constants $L_A, \; L_F, $ respectively. Assume that $A$ is a $\lambda$-strongly monotone operator and the couple ($A, F$) is $\gamma$-strongly monotone on $H.$ Then, the general variational inequality GVI($A, F, C$) admits a unique solution in $H,$ i.e., Sol($A, F, C$) = $\{x^*\}.$
\end{theorem}
\begin{proof} Since $A$ is $L_A$-Lipschitz continuous and $\lambda$-strongly monotone on $H,$ it possesses a continuous inverse $A^{-1}.$ The strong monotonicity of the couple ($A, F$) implies that $FA^{-1}$ is $\frac{\gamma}{L_A^2}$-strongly monotone and $\frac{L_F}{\lambda}$-Lipschitz continuous. According to Lemma \ref{existuniq}, the inverse variational inequality IVI($ FA^{-1}, C$ ) has a unique solution $u^*$, hence, GVI($A,; F;C$) admits a unique solution $x^* := A^{-1}u^*.$ $\blacksquare$
\end{proof}
\noindent{\bf Remark 2.2} Tangkhawiwetkul \cite{Tang} proposed a neurall network for the so-called generalized inverse mixed variational inequality with a strongly monotone couple of operators. In particular, based on the contraction mapping theorem, the author proved Theorem 2.2 under an additional assumptions:
$$ \sqrt{L_A^2 - 2\gamma + L_F^2} + \sqrt{1 - 2\lambda + L_A^2} < 1,$$
where $\gamma < \frac{L_A^2 + L_F^2}{2}; \; \lambda < \frac{1 + L_A^2}{2}.$\\
This additional condition narrows the class of GVI problems being studied.
\section{Regularized neurall network}
\noindent{Main assumptions}
\begin{assumption}\label{assumption1}We consider the problem \ref{GVI} under the following conditions:
\begin{itemize} 
				\item[\textbf{(A1)}] $A: H \to H $ is  $\lambda$- strongly monotone, $L_A$-Lipschitz continuous. $F$ is $L_F$-Lipschitz continuous and $A,   F$ are sequentially weak-to-weak continuous;
				\item[\textbf{(A2)}] the couple ($A, F$) is monotone;
				\item[\textbf{(A3)}] the solution set Sol($A, F, C$)  is nonempty. Moreover, the set $\mathcal C:= A({\rm Sol}(A, F, C))$ is convex.
\end{itemize}
\end{assumption}
\begin{remark}\label{CondA3}
Suppose $A: H \to H$ is a positive define, bounded linear operator, i.e., $\langle Ax, x\rangle \geq \ell\|x\|^2$   for all $x\in H,$ then $A$ satisfies Assumption (A1) with $L_A = \|A\|, \; \lambda = \ell.$  Moreover, $A$ is weak-weak continuous and the set $A({\rm Sol}(A, F, C))$  is convex if and only if  Sol($A, F, C$) is convex.
\end{remark}
We begin with the following obvious result, whose proof is based on the properties of the metric projection $P$ onto the closed convex set $C:$
\begin{lemma}\label{L3.1}
 $x^* \in$  Sol($A, F, C$) if and only if it is a solution to the equation
\begin{equation} \label{Eq1}
Fx^* = P(Fx^* - \mu Ax^*),
\end{equation}
where $\mu$ is an arbitrary fixed positive number.
\end{lemma}
In what follows, we assume the Assumption (A3) is satisfied. Clearly, Sol($A,F,C$) as well as $\mathcal{C}$ are closed subsets.\\
Since $A$ is $\lambda$-strongly monotone and $L_A$-Lipschitz continuous, its inverse $\mathcal{A}:= A^{-1}$ is $\frac{1}{\lambda}$-Lipschitz continuous and $\frac{\lambda}{L_A^2}$-strongly monotone.
Thus, the following bilevel variational inequality problem:
\begin{equation}\label{BVI}
\text{Find }u \in \mathcal{C}, \text{ such that, } \langle \mathcal{A}u, v - u \rangle \geq 0, \; \forall v \in \mathcal{C}  \tag{BVI}  
\end{equation}
admits a unique solution,  denoted by $u^\dagger$. Let $x^\dagger := \mathcal{A}u^\dagger \in {\rm Sol}(A,F,C).$
\vskip0.1cm
For each positive number $\alpha >0,$ we define the operator $F_\alpha := F + \alpha I,$ where $I$ is the identity operator in $H.$ Clearly, $F_\alpha$ is $(L_F + \alpha)$-Lipschitz continuous, and the couple ($A,  F_\alpha$) is $\alpha \lambda$-strongly monotone. According to Theorem 2.2, the so-called regularized general variational inequality GVI($A, F_\alpha, C$) admits a unique solution, denoted by $x_\alpha.$
\vskip0.1cm
{\bf Lemma 3.2}\label{lemma1h} Under Assumption \ref{assumption1}, the net $\{x_\alpha\}$ possesses the following properties:
\begin{itemize}
\item[(i)] boundedness on $(0, \infty);$ 
\item[(ii)] $\lim_{\alpha \to 0} x_\alpha = x^\dagger \in {\rm Sol}(A, F, C),$ where $x^\dagger=\mathcal A u^\dagger$ and $u^\dagger$ is the unique solution of the bilevel variational inequality \eqref{BVI};
\item[(iii)]  There exists $M > 0$ such that for all $\alpha; \beta  > 0,$ the estimate
$\|x_\alpha - x_\beta \| \leq \frac{M |\alpha -\beta|}{\beta}$ holds.
\end{itemize}
\begin{proof} (i) Let $x^* \in {\rm Sol}(A, F, C)$. Since $x_\alpha$ is the unique solution of the $\alpha\lambda$-strongly monotone inverse variational inequality GVI($A,F_\alpha$,C), according to Lemma \ref{lemma1}-(a), we have
		\begin{equation}\label{4}
			\left\langle F_\alpha x^*-  f, Ax^*-Ax_\alpha  \right\rangle +\left\langle  f-F_\alpha x_\alpha, Ax^*  \right\rangle  \geq \lambda \alpha \| x^*-x_\alpha  \|^2 \quad \forall f\in C.
		\end{equation}		 
		Taking $f=Fx^*$ in \eqref{4}, we get
		\begin{equation}\label{5}
			\alpha	\left\langle  x^*, Ax^*-Ax_\alpha  \right\rangle +\left\langle  Fx^* -F_\alpha x_\alpha, Ax^*  \right\rangle  \geq \lambda \alpha \| x^*-x_\alpha  \|^2.
		\end{equation}	
		Noting that $ \left\langle  Fx^* -F_\alpha x_\alpha, Ax^*  \right\rangle  \leq 0  $, from \eqref{5} we arrive at
		\begin{equation}\label{6}
			\left\langle  x^*, Ax^*-Ax_\alpha  \right\rangle  \geq \lambda \| x^*-x_\alpha  \|^2.
		\end{equation}
		Using the Lipschitz continuity of $A$, we get
		\begin{equation*}
			\| x^*-x_\alpha  \| \leq \frac {L_A}{\lambda } \|x^* \|.
		\end{equation*}
		Thus the net $\{x_\alpha  \}$ is bounded on $(0;\infty)$.\\
		(ii) Let $\hat x$ be a cluster point of $\{x_\alpha\}$.
        There exists a subsequence $\{x_{\alpha_k}\} \subset \{x_\alpha \}$ satisfying $ x_{\alpha_k} \rightharpoonup \hat x$. Using the definition of $x_{\alpha_k} \in {\rm Sol}(A, F_\alpha,C)$ and Lemma \ref{lemma1}-(a), we have
		\begin{equation}\label{7}
			\left \langle  Ay-Ax_{\alpha_k}, F y +\alpha_k y -f  \right \rangle + \left \langle  Ay, f-F x_{\alpha_k} -\alpha_k  x_{\alpha_k} \right \rangle \geq 0 \quad \forall y\in H,  f\in C.
		\end{equation}	
		In \eqref{7}, taking the limit as $k\to \infty$, using the weak-weak continuity of $A, F$, we get
		\begin{equation}\label{8}
			\left \langle  Ay-A\hat x, F y  -f  \right \rangle + \left \langle  Ay, f-F \hat x  \right \rangle \geq 0 \quad \forall y\in \mathcal  H,  f\in C.
		\end{equation}
        We have $F_{\alpha_k} x_{\alpha_k} \in C$ for all $k \geq 0$ and $F_{\alpha_k} x_{\alpha_k} =Fx_{\alpha_k}+\alpha_kx_{\alpha_k} \rightharpoonup F \hat x$. 
        Moreover, the set $C$ being closed and convex is also weakly closed. Hence,  
        \begin{equation}\label{haj1}
            F \hat x \in C.
        \end{equation}
        Lemma \ref{lemma1}-(b), together with \eqref{8} and \eqref{haj1} implies that $\hat x \in \text{Sol}(A,F,C)$. On the other hand, from \eqref{6}, it follows that
		\begin{equation}\label{hai1}
			\left\langle  x^*, Ax^*-A\hat x  \right\rangle  \geq 0 \quad \forall x^* \in {\rm Sol}(A,F,C).
		\end{equation}
Recall that the strongly monotone variational inequality problem VI($\mathcal A, \mathcal C$) admits a unique solution $u^\dagger$.\\
From \eqref{hai1}, we have $\hat u:= A \hat x$ is a solution of the problem
\begin{equation}\label{gf1}
\text{find } \hat u \in \mathcal C \text{ such that } \left\langle \mathcal{A} u,u- \hat u \right\rangle  \geq 0 \quad \forall u \in  \mathcal C.
\end{equation}
By the Minty lemma \cite[Lemma 1.5]{Kinderlehrer}, this problem is equivalent to the problem VI($\mathcal A, \mathcal C$), which admits a unique solution $u^\dagger$. Thus, $ \hat u \equiv u^\dagger,$   the net $\{x_\alpha\}$ has a unique weak cluster point and hence $x_\alpha \rightharpoonup \ x^\dagger =\mathcal A u^\dagger$. Finally, in \eqref{6}, taking $x^* = x^\dagger$ and letting $\alpha \to 0$, we get $x_\alpha \to  x^\dagger.$
        
		(iii)  From the definition of $x_\alpha$, we have
		\begin{equation}\label{9}
			\left\langle F_\beta x_\beta- F_\alpha x_\alpha  ,A  x_\alpha  \right\rangle  \geq 0
		\end{equation}	
		Similarly,
		\begin{equation}\label{10}
			\left\langle   F_\alpha x_\alpha -F_\beta x_\beta  ,A  x_\beta \right\rangle  \geq 0.
		\end{equation}	
		Adding \eqref{9} and \eqref{10} yields
		\begin{equation}\label{11}
			\left\langle F_\beta x_\beta- F_\beta x_\alpha  , A x_\alpha -Ax_\beta \right\rangle +
			\left\langle F_\beta x_\alpha- F_\alpha x_\alpha  , A x_\alpha -Ax_\beta \right\rangle 
			\geq 0.
		\end{equation}	
	Using the $\beta \lambda$-strong monotonicity of the couple $(A,F_\beta)$, we get
		\begin{equation}\label{12}
			(\beta-\alpha)	\left\langle  x_\alpha  , A x_\alpha -Ax_\beta \right\rangle 
			\geq \beta \lambda \|  x_\alpha -x_\beta \|^2.
		\end{equation}
        Denoting $K:=\sup\{\|x_\alpha\|: \alpha>0\}$ and $M:= \frac{L_A K}{\lambda},$ from \eqref{12}, we have
    \begin{equation}\label{12a}
    \|  x_\alpha -x_\beta \| \leq \frac{M}{\beta}|\beta -\alpha|,  \; \;\blacksquare
    \end{equation}  
	\end{proof}
According to Lemma \ref{L3.1}, each regularized solution $x_\alpha$ satisfies the relation
$$ P(F_\alpha x - \mu A x) - F_\alpha x = 0.$$
This observation leads us to the following regularized neural network
\begin{equation}\label{RegNN}
\begin{cases}
			& \dot x(t)=f(\alpha (t), \mu (t), x(t))\\
			& x(0)= x_0 \in  H,
\end{cases}
	\end{equation}
where $\alpha(t), \; \mu(t) : [0, \infty) \to (0, \infty)$ are  continuous functions and $f(\alpha (t), \mu (t), x(t)) := P\Big(F_{\alpha(t)} x(t) -\mu (t) Ax(t)  \Big)-F_{\alpha(t)} x(t)$.\\
Denote by $AC_{\mathrm{loc}}\left([0, +\infty), H \right) \; (L^1_{\mathrm{loc}}\left([0, +\infty), H \right)) $ the spaces of locally absolutely  continuous (locally integrable) functions, i.e., the spaces of all absolutely continuous (integrable) functions on each finite interval $[0, T], \; T >0.$\\	
A function $x: [0, +\infty) \to H$ is said to be a strong global solution of \eqref{RegNN} if the following properties hold:
	\begin{itemize}
		\item[a)] $x \in AC_{\mathrm{loc}}[0, +\infty),$ i.e., $x(t)$ is absolutely continuous on each interval $[0, T],$ for any $0<T< +\infty;$
		\item[b)] For almost everywhere $t \in [0, +\infty)$ equation \eqref{RegNN} holds;
		\item[c)] $x(0) = x_0.$
	\end{itemize}
\begin{theorem}\label{Th3.3}
 Let $\alpha(t), \; \mu(t): [0, \infty) \to (0, \infty)$ be given positive continuous functions. Under Assumption \ref{assumption1}, the neural network \eqref{RegNN} admits a unique solution.
 \end{theorem}
 \begin{proof}
 First observe that the two functions $\alpha \longmapsto f(\alpha, \mu, x):= P(Fx+\alpha x- \mu Ax) - (Fx + \alpha x)$ and $\mu \longmapsto f(\alpha, \mu, x)$ are continuous on $[0, +\infty).$ Then the mapping $f(\alpha(t),\mu(t),\cdot): H \to H$ for all  $t\in [0, +\infty)$ is continuous. Moreover, by the nonexpansiveness of $P$ and Lipschitz continuity of $A, \; F,$ we get
\mygreen{ 	$$\forall x, y \in H, \; \|f(\alpha(t),\mu(t), x)- f(\alpha(t),\mu(t), y)\| \leq M(t)\|x-y\|,$$}
    where 
$M(t):= 2L_F + 2 \alpha(t) + \mu(t)L_A \in C[0, +\infty).$	\\	
Besides, the function   $t\longmapsto \varphi(t):= \|f(\alpha(t), \mu(t), x^\dagger)\|,$ where $x^\dagger$ is the unique solution of the bilevel problem BVI($A,F,C$), is continuous on $[0, +\infty).$ \\
For any $x \in H,$ we have  
\begin{align*}
 \|f(\alpha(t), \mu(t), x)\| & \leq \|f(\alpha(t), \mu(t), x^\dagger)\| + \|f(\alpha(t), \mu(t), x) - f(\alpha(t), \mu(t), x^\dagger)\| \\
& \leq \varphi(t) + M(t) \|x - x^\dagger\| \\
&\leq \varphi(t)+M(t)\|x^\dagger\|+M(t)\|x\| \\
&\leq \sigma(t)(1 + \|x\|),
\end{align*}	
where the function $\sigma(t) := \max\{\varphi(t) + M(t)\|x^\dagger\|, M(t)\}$ is also continuous on $[0, +\infty),$  hence $\sigma \in L^1_{\mathrm{loc}}[0, +\infty). $\\
Finally, from the fact that the mapping $f(\alpha(t),\mu(t),\cdot) $ is continuous on $[0, +\infty)$ and the estimate $ \|f(\alpha(t), \mu(t), x)\| \leq \sigma(t)(1 + \|x\|),$ with $\sigma \in L^1_{\mathrm{loc}}[0, +\infty),$ holds for all $x\in H,$ we conclude that 
$$\forall x \in H, \; f(\cdot, \cdot, x) \in L^1_{\mathrm{loc}}\left([0, +\infty), H\right).$$
Thus, by the Cauchy-Lipschitz-Picard theorem \cite[Prop. 6.2.1]{Haraux}, the Cauchy problem \eqref{RegNN} admits a unique global solution. $\blacksquare$
\end{proof}
\begin{theorem}\label{Theorem_continuous_convergence}
Let $x(t)$ be a strong global solution of \eqref{RegNN}. Suppose that Assumption \ref{assumption1} is satisfied,  $A$ is a linear, self-adjoint operator and that 
$$
\lim_{t\to \infty} \alpha(t) = 0, \; \lim_{t\to \infty} \alpha(t) \mu(t) =\infty,\quad \lim_{t\to \infty} \frac{\dot \alpha (t)}{\alpha^2 (t)} =0,\quad \bigintssss_0^\infty \alpha(t) dt=\infty, 
$$
where $\alpha(t)$ is a given positive continuously differentiable function and $\mu(t)$ is a chosen positive continuous function. Then, $\|x(t) - x^\dagger \| \to 0,$  as $t \to \infty.$ 
\end{theorem} 
\begin{proof}
According to Lemma \ref{lemma1h}(ii), $\|x_{\alpha(t)}- x^\dagger \| \to 0,$ as $t \to \infty,$ we only need to prove that $\lim\limits_{t\to\infty} \| x(t)-x_{\alpha(t)} \|=0$.  Denote
$$
V(t):=\frac 1 2 \left\langle A x(t)-Ax_{\alpha(t)},x(t)-x_{\alpha(t)} \right\rangle.  
$$
Since $ V(t)\geq \frac 1 2 \lambda \|x(t)-x_{\alpha(t)} \|^2 $,  it remains to show that $V(t) \to 0 $ as  as $t \to \infty.$ We have
\begin{align}\label{rf1}
    \dot V(t) &= \frac 1 2 \left\langle A x(t)-Ax_{\alpha(t)}, \dot x(t) +\dfrac{d}{dt} x_{\alpha(t)} \right\rangle  +\frac 1 2  \left\langle \dfrac{d}{dt} A x(t)+\dfrac{d}{dt}Ax_{\alpha(t)} , x(t)-x_{\alpha(t)}  \right\rangle\notag\\  
    &= \left\langle A x(t)-Ax_{\alpha(t)}, \dot x(t) +\dfrac{d}{dt} x_{\alpha(t)} \right\rangle.
\end{align}
In the last equality, we use the assumption that $A$ is a linear, self-adjoint operator.
By Lemma \ref{lemma1h}-(iii), the regularized solution $x_{\alpha(t)} \in {\rm Sol}(A, F_{\alpha(t)}, C)$ satisfies the relation 
$$\|x_{\alpha(t)} - x_{\alpha(s)}\| \leq M \left|\frac{\alpha(t)- \alpha(s)}{\alpha(t)}\right|, \; \forall t, s \in [0, +\infty), $$
hence, it is absolutely continuous due to the absolute continuity of $\alpha(t).$ Therefore, it is differentiable almost every where in $t.$ Obviously, 
\begin{equation}\label{1Th3}    
			\left\|  \frac  d {dt} x_{\alpha(t)}  \right\| \leq M   \left|  \frac {\dot \alpha(t)}{\alpha(t)}  \right|.
		\end{equation}  
Denote $y(t):=P\Big(F_{\alpha(t)} x(t) -\mu (t) Ax(t)  \Big)$. Using the property of the projection, we have
$$
\left\langle  u-y(t),F_{\alpha(t)} x(t)-\mu(t) A x(t) -y(t) \right\rangle   \leq 0 \quad \forall u \in C.
$$
Taking $u=F_{\alpha(t)} x _{\alpha(t)}   $, we obtain
$$
\left\langle  y(t)-F_{\alpha(t)} x _{\alpha(t)} ,F_{\alpha(t)} x(t)-\mu(t) A x(t) -y(t) \right\rangle \geq 0.
$$
On the other hand, since $x_{\alpha(t)} $ solves the problem GVI$(A,F _{\alpha(t)},C )$, it holds that
$$
\mu(t) \left\langle A x _{\alpha(t)} ,y(t)-F _{\alpha(t)} x _{\alpha(t)}  \right\rangle  \geq 0.
$$
Combining the last two inequalities, we arrive at
$$
\left\langle y(t)-F _{\alpha(t)} x _{\alpha(t)},F _{\alpha(t)} x-\mu(t) [A x(t)- A x _{\alpha(t)} ]-y(t)  \right\rangle  \geq 0.
$$
Thus,
\begin{align}\label{re3}
    &\mu(t) \left\langle A x(t) -A x _{\alpha(t)} ,y(t) -F _{\alpha(t)} x(t) \right\rangle \leq \left\langle  F _{\alpha(t)} x(t)-y(t),F _{\alpha(t)} x(t) -F _{\alpha(t)} x _{\alpha(t)}  \right\rangle+ \notag\\
    & -\mu(t) \left\langle A x(t) -A x _{\alpha(t)} , F _{\alpha(t)} x(t) - F _{\alpha(t)}  x _{\alpha(t)}  \right\rangle - \|y(t) -F _{\alpha(t)} x(t)   \| ^2 \notag\\
    & \leq -\mu(t) \alpha(t) \lambda \| x(t)-x _{\alpha(t)}   \| ^2 -\frac 1 2 \|y(t) -F _{\alpha(t)} x(t)   \| ^2+ \frac 1 2 \|F _{\alpha(t)} x(t) -F _{\alpha(t)} x _{\alpha(t)}    \| ^2   \notag\\
    &\leq - \left( \mu(t) \alpha(t) \lambda -\frac 1 2 (L_F+\alpha(t))^2 \right) \|  x(t)-x _{\alpha(t)} \| ^2 - \frac 1 2 \|y(t) -F _{\alpha(t)} x(t)   \| ^2.
\end{align}
Combining \eqref{rf1} and \eqref{re3}, we have
\begin{align}\label{re4}
V(t)&\leq     \left\langle A x(t) -A x _{\alpha(t)} ,y(t) -F _{\alpha(t)} x(t) \right\rangle + ML_A\left| \frac{\dot \alpha(t)}{\alpha(t)}\right|  \| x(t)-x _{\alpha(t)}   \| \notag\\
&\leq - \left(  \alpha(t) \lambda -\frac 1 {2\mu(t)} (L_F+\alpha(t))^2 \right) \|  x(t)-x _{\alpha(t)} \| ^2 - \frac 1 {2\mu(t)} \|y(t) -F _{\alpha(t)} x(t)   \| ^2+\notag\\
&+ML_A\left| \frac{\dot \alpha(t)}{\alpha(t)}\right|  \| x(t)-x _{\alpha(t)}   \|.
\end{align}
Since $\alpha (t) \mu (t) \to +\infty$ as $t\to \infty$, without loss of generality, we may assume that $\alpha(t) \lambda -\frac 1 {2\mu(t)} (L_F+\alpha(t))^2 >0$ for all $t\geq 0.$ Thus,
using the Lipschitz continuity and the strong monotonicity of $A$, from \eqref{re4} we get
\begin{align}\label{ret2}
    \dot V(t) \leq - \left(  \frac {2\alpha(t) \lambda} {L_A} -\frac {(L_F+\alpha(t))^2} {L_A\mu(t)}  \right) V(t)  + ML_A\left| \frac{\dot \alpha(t)}{\alpha(t)}\right|\sqrt{\frac{2V}{\lambda}}.
\end{align}
From this differential inequality, we obtain
$$
\sqrt{V(t)} \leq \frac{\bigints_0^t \left\{ \exp \left\{ \bigintss_0^u  \left(  \frac {\alpha(s) \lambda} {L_A} -\frac {(L_F+\alpha(s))^2} {2L_A\mu(s)}  \right) ds  \right\} ML_A\left| \frac{\dot \alpha(u)}{\alpha(u)}\right|\sqrt{\frac{1}{2\lambda}} \right\}du +\sqrt{V(0)}  }{  \exp \left\{ \bigints_0^t  \left(  \frac {\alpha(u) \lambda} {L_A} -\frac {(L_F+\alpha(u))^2} {2L_A\mu(u)}  \right) du  \right\} }.
$$
Since $\alpha(t) \mu(t) \to \infty$ as $t\to \infty$, there exists $t_0>0$ such that
$$
 \frac {\alpha(u) \lambda} {L_A} -\frac {(L_F+\alpha(u))^2} {2L_A\mu(u)}   \geq  \frac {\alpha(u) \lambda} {2L_A} \quad \forall u\geq t_0. 
$$
Moreover, we have $\int_0^{\infty} \alpha (t) dt =\infty$, hence 
$$
\lim_{t\to \infty}  \exp \left\{ \bigintss_0^t  \left(  \frac {\alpha(u) \lambda} {L_A} -\frac {(L_F+\alpha(u))^2} {2L_A\mu(u)}  \right) du  \right\} =\infty.
$$
Now, it is sufficient to show that 
$$
L:=\lim_{t\to \infty} \frac{\bigints_0^t \left\{ \exp \left\{ \bigintss_0^u  \left(  \frac {\alpha(s) \lambda} {L_A} -\frac {(L_F+\alpha(s))^2} {2L_A\mu(s)}  \right) ds  \right\} ML_A\left| \frac{\dot \alpha(u)}{\alpha(u)}\right|\sqrt{\frac{1}{2\lambda}} \right\}du}{\exp \left\{ \bigints_0^t  \left(  \frac {\alpha(u) \lambda} {L_A} -\frac {(L_F+\alpha(u))^2} {2L_A\mu(u)}  \right) du  \right\}}=0.
$$
If the numerator is finite, we get the desired result. In the opposite case, applying L'hospital's rule, we have
\begin{align*}
    L&=\lim_{t\to \infty} \frac{  \exp \left\{ \bigints_0^t  \left(  \frac {\alpha(s) \lambda} {L_A} -\frac {(L_F+\alpha(s))^2} {2L_A\mu(s)}  \right) ds  \right\} ML_A\left| \frac{\dot \alpha(t)}{\alpha(t)}\right|\sqrt{\frac{1}{2\lambda}} }{\exp \left\{ \bigints_0^t  \left(  \frac {\alpha(u) \lambda} {L_A} -\frac {(L_F+\alpha(u))^2} {2L_A\mu(u)}  \right) du  \right\} \left(  \frac {\alpha(t) \lambda} {L_A} -\frac {(L_F+\alpha(t))^2} {2L_A\mu(t)}  \right) }\\
    &=\lim_{t\to \infty} \frac{ML_A^2\dot \alpha(t)}{\lambda \sqrt{2\lambda}\alpha^2(t)}\\
    &=0
\end{align*}
and this completes the proof.  $\blacksquare$
\end{proof}
\begin{remark}
In Theorem \ref{Theorem_continuous_convergence}, one can chose $\alpha(t) = (1+t)^{-p}$ and $\mu(t) = (1+t)^{q},$  where $0<p< q< 1.$

\end{remark}
Next, we assume that $(A,F)$ is $\gamma$-strongly monotone. In this case, we do not need to regularize the mapping $F$ and  \eqref{RegNN}  becomes
\begin{equation}\label{RegNN2}
\begin{cases}
\dot x(t)=P(Fx(t)-\mu A x(t) )-Fx(t),\\
x(0)\in H.
\end{cases}
\end{equation}
We will establish the convergence rate of the solution $x(t)$ of \eqref{RegNN2} to the unique solution $x^\dagger$ of \ref{GVI}. 
\begin{theorem}\label{Theo_rate}
    Let $x(t)$ be a strong global solution of \eqref{RegNN2}.
    Suppose that $A,F$ satisfy all the conditions in Theorem \ref{Theorem_continuous_convergence}, moreover,  $(A,F)$ is $\gamma$-strongly monotone and $\mu > \frac 1 {2\gamma} L_F^2$. Then, we have
    $$
    \|x(t) -x^\dagger   \|  \leq \sqrt{\frac{L_A}{\lambda}}  \| x(0)-x^\dagger  \| e^{-ct}, 
    $$
    where $c=\gamma -\frac 1 {2\mu} L_F^2>0$.
\end{theorem}
\begin{proof} 
Denote 
$$
W(t)=\frac 1 2 \left\langle Ax(t)-Ax^\dagger,x(t)-x^\dagger \right\rangle   , y(t)=P(Fx(t)-\mu Ax(t)). 
$$
Similarly to \eqref{re3}, we obtain
\begin{align}\label{re3a}
     \left\langle A x(t) -A x ^\dagger ,y(t) -F x(t) \right\rangle\leq - \left(  \gamma -\frac 1 {2\mu} L_F^2 \right) \|  x(t)-x^\dagger \| ^2 - \frac 1 {2\mu} \|y(t) -F x(t)   \| ^2.
\end{align}
Thus, instead of \eqref{re4}, we obtain
\begin{align}\label{re4a}
\dot W(t)&\leq- \left(  2 \gamma -\frac 1 {\mu} L_F^2 \right) W(t) - \frac 1 {2\mu} \|y(t) -F  x(t)   \| ^2.
\end{align}
It implies that
$$
W(t)\leq W(0) e^{-\left( 2 \gamma -\frac 1 {\mu} L_F^2\right)t}.
$$
Since $A$ is $\lambda$-strongly monotone and $L_A$-Lipschitz continuous, we have
$$
  \lambda \| x(t)-x^\dagger  \| ^2  \leq \left\langle Ax(t)-A x^\dagger,x(t)-x^\dagger \right\rangle  \leq L_A \| x(t)-x^\dagger  \| ^2 
$$
and hence
$$
\| x(t)-x^\dagger  \| \leq \sqrt{\frac{L_A}{\lambda}}  \| x(0)-x^\dagger  \| e^{-ct},
$$
where $c:= \gamma -\frac 1 {2\mu} L_F^2 >0$. \; $\blacksquare$
\end{proof}
\section{Discretization of the neural network}
An explicit Euler scheme applied to the Cauchy problem \eqref{RegNN} gives
\begin{equation}\label{Discret}
		\begin{cases}
			& y^{k}=P \Big(F_{\alpha_k} x^k -\mu_k Ax^k  \Big)\\
			&x^{k+1} = x^k +h_k (y^k-F_{\alpha_k}x^k  )\\
			& x^0 \in H.
		\end{cases}
	\end{equation}
\begin{theorem}\label{theo2discret}
   Suppose that $A,F$ satisfy all the conditions in Theorem \ref{Theorem_continuous_convergence} and 
   $$
   \sum_{k=0}^\infty h_k\alpha_k =\infty;  \lim_{k\to\infty} \mu_k h_k =0;  \lim_{k\to\infty} \mu_k \alpha_k =\infty; \lim_{k\to\infty} \frac{|\alpha_k -\alpha_{k+1} |}{h_k \alpha_k^2} =0; \lim_{k\to\infty} \alpha_k=0.
   $$
Then the sequence $\{x^k\},$ generated by \eqref{Discret} converges to the  solution $x^\dagger$ of the GVI($A, F, C$).
\end{theorem}
\begin{proof}
Similarly to \eqref{re3}, we get
\begin{equation}\label{fe1}
         \left\langle A x^k -A x _{\alpha_k} ,y^k -F _{\alpha_k} x^k \right\rangle    \leq - \left(  \alpha_k \lambda -\frac 1 {2\mu_k} (L_F+\alpha_k)^2 \right) \|  x^k-x _{\alpha_k} \| ^2 - \frac 1 {2\mu_k} \|y^k -F _{\alpha_k} x^k   \| ^2.
\end{equation}
Denote $U_{k}:=\left\langle Ax^{k}-A x_{\alpha_{k}},x^k-x_{\alpha_k}\right\rangle$.   
We have
\begin{align}\label{aq1}
    U_{k+1}-U_k&=\left\langle A x^{k+1}-Ax _{\alpha_{k+1} }  , x^{k+1}-x^k -(x _{\alpha_{k+1}}-x _{\alpha_{k}}  )  \right\rangle+  \notag \\
    &+ \left\langle  A x^{k+1}-A x^k-(Ax _{\alpha_{k+1} }  -A x _{\alpha_{k} }  ),x^k -x _{\alpha_{k}}   \right\rangle.
\end{align}
Since $A$ is a linear, self-adjoint  and Lipschitz continuous operator, it holds that
\begin{align}\label{aq2}
  & \left\langle A x^{k+1}-Ax _{\alpha_{k+1} }  , x^{k+1}-x^k -(x _{\alpha_{k+1}}-x _{\alpha_{k}}  )  \right\rangle \leq  \left\langle A x^k -A x _{\alpha_k} ,x^{k+1}-x^k  \right\rangle +\notag \\ 
  & + L_A \|  x^{k+1}-x _{\alpha_{k+1} }   \| \| x _{\alpha_{k+1}}-x _{\alpha_{k}}   \| +  L_A \| x^{k+1}-x^k  \|^2 +L_A \| x^{k+1}-x^k  \| \|  x _{\alpha_{k+1}}-x _{\alpha_{k}} \| \notag \\ 
  &\leq \left\langle A x^k -A x _{\alpha_k} ,x^{k+1}-x^k  \right\rangle + 2L_A \| x^{k+1}-x^k  \| \|x _{\alpha_{k+1}}-x _{\alpha_{k}}     \|+ L_A \| x _{\alpha_{k+1}}-x _{\alpha_{k}}   \| ^2 + \notag \\ 
  &+L_A \| x^k -x _{\alpha_{k}}   \|   \| x _{\alpha_{k+1}} -x _{\alpha_{k}}   \| + L_A \|  x^{k+1}-x^k \| ^2. 
\end{align}
and
\begin{align}\label{aq3}
   \left\langle  A x^{k+1}-A x^k-(Ax _{\alpha_{k+1} }  -A x _{\alpha_{k} }  ),x^k -x _{\alpha_{k}}   \right\rangle & \leq  \left\langle A x^k -A x _{\alpha_k} ,x^{k+1}-x^k  \right\rangle + \notag \\ 
   &+L_A \|  x _{\alpha_{k+1}}-x _{\alpha_{k}} \| \| x^k-x _{\alpha_{k}}   \| .
\end{align}
Combining \eqref{fe1}, \eqref{aq1}, \eqref{aq2} and \eqref{aq3}, we get
\begin{align}\label{aq4}
  &  U_{k+1}-U_k \leq 2\left\langle A x^k -A x _{\alpha_k} ,x^{k+1}-x^k  \right\rangle + 2L_A \| x^{k+1}-x^k  \| \|x _{\alpha_{k+1}}-x _{\alpha_{k}}     \| + \notag \\ 
  &+2L_A \| x^k -x _{\alpha_{k}}   \|   \| x _{\alpha_{k+1}} -x _{\alpha_{k}}   \| + L_A \|  x^{k+1}-x^k \| ^2+ L_A \| x _{\alpha_{k+1}}-x _{\alpha_{k}}   \| ^2 \notag \\
  & \leq - \left( 2h_k \alpha_k \lambda -\frac {h_k} {\mu_k} (L_F+\alpha_k)^2 \right) \|  x^k-x _{\alpha_k} \| ^2 - \frac {h_k} {\mu_k} \|y^k -F _{\alpha_k} x^k   \| ^2 +  \notag \\ 
  &+2L_A \| x^k -x _{\alpha_{k}}   \|   \| x _{\alpha_{k+1}} -x _{\alpha_{k}}   \| + L_A \|  x^{k+1}-x^k \| ^2+ L_A \| x _{\alpha_{k+1}}-x _{\alpha_{k}}   \| ^2+ \notag \\
  &+2L_A \| x^{k+1}-x^k  \| \|x _{\alpha_{k+1}}-x _{\alpha_{k}}     \|. 
\end{align}
Since $\mu_k \alpha_k \to \infty$ as $k\to \infty$, without loss of generality, we may assume that 
$$
2h_k \alpha_k \lambda -\frac {h_k} {\mu_k} (L_F+\alpha_k)^2 \geq h_k \alpha_k \lambda\quad \forall k\geq 0.
$$
On the other hand, applying the Cauchy inequality, we get
$$
2L_A \| x^k -x _{\alpha_{k}}   \|   \| x _{\alpha_{k+1}} -x _{\alpha_{k}}   \| \leq  \frac 1 2  h_k \alpha_k \lambda \| x^k -x _{\alpha_{k}}   \| ^2 + \frac {2L_A^2} {h_k \alpha_k \lambda}\| x _{\alpha_{k+1}} -x _{\alpha_{k}}   \|^2.
$$
and 
$$
2L_A \| x^{k+1}-x^k  \| \|x _{\alpha_{k+1}}-x _{\alpha_{k}}     \| \leq L_A \| x^{k+1}-x^k  \| ^2 + L_A \|x _{\alpha_{k+1}}-x _{\alpha_{k}}     \|^2
$$
Thus,
\begin{align}\label{aq5}
    U_{k+1}-U_k  & \leq -\frac 1 2  h_k \alpha_k \lambda  \|  x^k-x _{\alpha_k} \| ^2 - \left( \frac {h_k} {\mu_k} -2L_Ah_k^2 \right) \|y^k -F _{\alpha_k} x^k   \| ^2 +  \notag \\ 
  &+   \left( 2L_A +\frac {2L_A^2} {h_k \alpha_k \lambda} \right) \| x _{\alpha_{k+1}}-x _{\alpha_{k}}   \| ^2 \notag \\
  & \leq -\frac 1 2  h_k \alpha_k \lambda  \|  x^k-x _{\alpha_k} \| ^2 - \left( \frac {h_k} {\mu_k} -2L_Ah_k^2 \right) \|y^k -F _{\alpha_k} x^k   \| ^2 +  \notag \\ 
  &+   M^2 \left( 2L_A +\frac {2L_A^2} {h_k \alpha_k \lambda} \right)\frac {| \alpha_{k+1}-\alpha_{k}|^2}{\alpha_k^2}.
\end{align}
Since $\mu_k h_k \to 0$, without loss of generality, we may assume that $\frac {h_k} {\mu_k} -2L_Ah_k^2 \geq 0  $ for all $k\geq 0.$ On the other hand, since $A$ is $L_A$-Lipschitz continuous, we have $ U_k \leq L_A \| x^k-x_{\alpha_k} \|^2$. Combining these facts with \eqref{aq5},  we obtain
\begin{align}\label{aq6}
    U_{k+1}  & \leq \left(1-\frac {h_k \alpha_k \lambda} {2L_A}  \right)  U_k +   M^2 \left( 2L_A +\frac {2L_A^2} {h_k \alpha_k \lambda} \right)\frac {| \alpha_{k+1}-\alpha_{k}|^2}{\alpha_k^2}.
\end{align}
Since
$$
\sum_{k=0}^\infty h_k \alpha_k =\infty,\quad \lim_{k\to \infty} \frac{ | \alpha_{k+1}-\alpha_{k}|^2 }{\alpha_k^4 h_k^2}=0,
$$
applying Lemma \ref{lemma31} with 
$$
\theta_k:= \frac {h_k \alpha_k \lambda} {2L_A};\quad \zeta_k := M^2 \left( 2L_A +\frac {2L_A^2} {h_k \alpha_k \lambda} \right)\frac {| \alpha_{k+1}-\alpha_{k}|^2}{\alpha_k^2},
$$
we get $U_k\to 0$. Finally, since $A$ is strongly monotone, we have $U_k \geq \lambda \| x^k-x_{\alpha_k}\|^2$, and therefore, $\| x^k-x_{\alpha_k}\| \to 0.$ According to Lemma \ref{lemma1h}-(ii), $x_{\alpha_k} \to x^\dagger$ as $k \to\infty,$ hence $\lim\limits_{k\to \infty} x^k = x^\dagger.$
$\blacksquare$
\end{proof}
\begin{remark}\label{remark1h}
    An example of $\{\alpha_k\}$, $\{\mu_k\}$, $\{h_k\}$ satisfying the conditions in Theorem \ref{theo2discret} is
    $$
    \alpha_k=\frac 1 {(k+1)^p},\ \mu_k=(k+1)^q,\ h_k=\frac 1 {(k+1)^r},\quad   0<p<q<r, p+r<1.
    $$
\end{remark}
\section{Numerical experiments}
\begin{example}
		Let $ H = \ell^2$ and $m\geq 2$ be a fixed positive integer. For any $x = (x_1, x_2, \ldots) \in \ell^2,$ we define the operators $A x = (x_1, x_2,\frac{3+1}{3}x_3 , \ldots, \frac{n+1}{n}x_n, \ldots)$ and $Fx = (-x_2, x_1, 0, \ldots ).$ It is easy seen that $A$ is Lipschitz continuous and strongly monotone with $L_A = 2, \; \mu =1.$ Further, 
$F$ is a Lipschitz continuous operator with $L_F = 1.$ Let $\ell^2 \ni x^n =(x_1^n, x_2^n, \ldots) \rightharpoonup x=(x_1, x_2, \ldots) $ and $u = (1,0, \ldots), \; v =(0, -1, 0, \ldots).$ Then $x_n^1 = \langle x^n, u \rangle  \to \langle x, u \rangle = x_1, \;\; -x_2^n = \langle x^n,v \rangle \to \langle x, v\rangle = -x_2 \; (n \to \infty). $ The fact $F(x^n)= (-x_2^n, x_1^n, 0, \ldots) \to (-x_2, x_1, 0, \ldots) = F(x)$ as $n \to \infty,$ shows that $F$ is sequentially weak-to-strong continuous, hence it is weak-to-weak continuous. Finally, for all $x=(x_1, x_2, \ldots, x_n, \ldots)$ and $y=(y_1, y_2, \ldots, y_n, \ldots),$  one has 
$\langle Ax - Ay, Fx -Fy\rangle = 0,$ i.e., the couple ($A, F$) is monotone.\\
Define a set $C= \{x \in \ell^2: x_i = 0, \; i=1, \ldots, m\}$.\\ 
		Observe that $Fx \in C$ if and only if $x = (0, 0, x_3,\ldots).$ For such $x$ and for any $y \in C,$ we have $F x =0 ,$ and
		$$T(x, y):= \langle Ax, y - Fx\rangle = \langle Ax, y \rangle=   \sum_{j= m+1}^\infty \frac{j+1}{j}x_jy_j. $$
		We show that $T(x, y) \geq 0$ for all $y\in C$ if and only if $x_j = 0, \; j \geq m+1.$ \\
		Indeed, if the latter condition holds then $T(x, y) =0$ trivially for all $y \in C.$\\ Conversely,  suppose that there exists $\tilde{x}=(0,0,\tilde{x}_3, \ldots, \tilde{x}_j, \ldots)$ 
		such that $\tilde{x}_j \neq 0$ for some $ j \geq m+1,$ while  $T(\tilde{x},y) \geq 0$ for all $y \in C.$ 
		Taking $\tilde{y} \in C$ with $\tilde{y}_j := - \tilde x_j $ and  $\tilde{y}_k := 0$ for all $k\neq j$. Then,   $T(\tilde{x}, \tilde{y})= -(\tilde{x}_j)^2 <0,$ which is a contradiction. Hence the solution set Sol($A, F, C$) consists of all $x=(0, 0, x_3, \ldots, x_m, 0, \ldots).$ Clearly, Sol($A, F, C$) is nonempty, closed, convex subset in $\ell^2.$ According to Remark \ref{CondA3}, $A$ is weak-weak continuous, and the subset $\mathcal{C}:= A({\rm Sol}(A,F,C))$ is convex.\\
Thus, all the assumptions  $(A_1)-(A_3)$ are fulfilled.\\
Clearly, $\mathcal{A}y = (y_1, y_2, \frac{3}{4}y_3,\ldots, \frac{n}{n+1}y_n, \ldots).$ Let $u^\dagger:= (0, 0, u_3, \ldots,u_m, 0, \ldots) \in \mathcal{C}$ be a unique solution to the BVI
$$\langle \mathcal{A}u^\dagger, v- u^\dagger \rangle \geq 0, \; \forall v \in \mathcal{C}.$$
Letting $v =0 \in \mathcal{C},$ we get
$$-\langle \mathcal{A}u^\dagger, u^\dagger \rangle = -\frac{3}{4}u_3^2-\ldots-\frac{m}{m+1}u_m^2 \geq 0. $$
Thus $u^\dagger = 0,$ hence $x^\dagger = \mathcal{A} u^\dagger =0.$\\
The regularized neural network \eqref{RegNN} in this case is of the form:
\begin{equation}\label{RegNN1}
			\begin{cases}
				&\dot{x}_1(t) = x_2(t) - \alpha(t) x_1(t),\\
				&\dot{x}_2(t) = -x_1(t) - \alpha(t) x_2(t),\\
				&\dot{x}_k(t)=  - \alpha(t) x_k(t), \; k= 3,\ldots,m ,\\
				&\dot{x}_j(t) = - \frac{j}{J+1}\mu(t) x_j(t), \; j \geq m+1,\\
				& x(0) = x^0 \in \ell^2, 
			\end{cases}
		\end{equation}
Integrating \eqref{RegNN1} gives
\begin{equation*}
			\begin{cases}
				&{x}_1(t) =     e^ {-\int_0^t \alpha (u)du } (x^0_1 \cos t +x^0_2 \sin t ),\\
				&{x}_2(t) =     e^ {-\int_0^t \alpha (u)du } (x^0_2 \cos t -x^0_1 \sin t ),\\
				&x_k(t)=  x^0_k e^ {-\int_0^t \alpha (u)du }, \; k= 3,\ldots,m ,\\
				&x_j(t) = x^0_j \frac{j+1}{j}e^ {-\int_0^t \mu (u)du }, \; j \geq m+1,\\
			\end{cases}
		\end{equation*}
Further, 
\begin{equation*}
			\|x(t)- x^\dagger\|^2 = \|x(t)\|^2 = \sum\limits_{i=1}^m|x_i(0)|^2 e^{-2\int_0^t \alpha(u){\rm du}}
			+ \sum\limits_{j=m+1}^\infty |x_j^0|^2 \Big(\frac{j+1}{j}\Big)^2e^{-2\int_0^t\mu(u) \mathrm{du}}
		\end{equation*} 		
The last estimate shows that
$$\|x(t) - x^\dagger\|  \leq \sum\limits_{i=1}^m|x_i(0)|^2 e^{-2\int_0^t \alpha(u){\rm du}} + 
4 \Big(\|x(0)\|^2 - \sum\limits_{i=1}^m|x_i(0)|^2 \Big)e^{-2\int_0^t \mu(u) \mathrm{du}}.$$
If $\int_0^\infty \alpha(u){\rm du} = \int_0^\infty \mu(u){\rm du} = +\infty$  then $\|x (t) - x^\dagger \| \to 0$ as $t$ approaches $\infty.$
\vskip0.1cm
\end{example}
\begin{example}\label{ex2}
Let $H = \mathbb{R}^3,\; C = B[0, 1]$- a closed unit ball centered at $0.$ Define a nonsingular matrix
$$ B=\begin{pmatrix}
1 & 0 & 0\\
1 & 1 & 0\\
1 & 1 & 1\\ 
\end{pmatrix} $$
and let $A = B^TB,$ then $A$ is a symmetric positive define matrix. A simple calculation shows that 
$$ A=\begin{pmatrix}
1 &1 &1\\
1 & 2 & 2\\
1 & 2 & 3 \\ 
\end{pmatrix},\,\,\,\,\,\, A^{-1} = \begin{pmatrix}
2 & -1 &0\\
-1 & 2 &-1\\
0 & -1 & 1\\ 
\end{pmatrix} $$
Let $Fx:=A^{-1}P x,$ where $Px$ is a metric projection of $x$ onto $C.$ Then,
\begin{equation*}
Fx = \left \{
\begin{array}{ll}
 A^{-1} x, \; \; \|x\| \leq 1, \\
\frac{A^{-1}x}{\|x\|} \; \;\; \|x\| >1
\end{array}
\right .
\end{equation*}
Since the metric projection $P$ is firmly nonexpansive and the matrix $A$ is symmetric, one gets
$$\langle Ax- Ay, Fx - Fy\rangle = \langle  x- y, Px - Py\rangle \geq \|P x -P y\|^2 \geq 0, $$ 
for all $x, y \in H,$ which yields the monotonicity of the couple ($A, F$).\\
Clearly, $x = 0$ is a solution of the GVI($A, F, C$). Suppose $x^* $ is an another solution, then $Fx^* \in C, \; \langle Ax^*, y - Fx^*\rangle \geq 0 $ for all $y \in C.$ In particular, setting $y = 0,$ we find $-\langle Ax^*, Fx^*\rangle \geq 0,$ or equivalently, $\langle x^*, Px^*\rangle \leq 0,$ which implies that $x^* =0.$ Thus Sol($A, F, C$) $= \{0\}$ and $x^\dagger = 0.$
Applying the algorithm \eqref{Discret}, we get 
\begin{equation}\label{Discret1}
		\begin{cases}
        &z^k = A^{-1}P x^k\\
			& y^{k}=P \Big(z^k + \alpha_k x^k -\mu_k Ax^k  \Big)\\
			&x^{k+1} = x^k +h_k (y^k-z^k - \alpha_k x^k  )\\
				\end{cases}
	\end{equation}
\vskip0.1cm
The parameters of the algorithms are chosen as follows:
\begin{itemize}
    \item[(i)] In Algorithm \eqref{Discret}, $\alpha_k=\frac 1 {(k+1)^p}$, $\mu_k=(k+1)^q,$ $h_k=\frac 1 {(k+1)^r}$. We test the algorithm with different $p,q,r$ satisfying $0<p<q<r, p+r<1$.  
    \item[(ii)] In the dynamical system  \eqref{RegNN}, $\alpha (t)=(t+1)^{-0.2} $, $\mu (t)=(t+1)^{0.4}.$
\end{itemize}
The starting point are randomly generated.
The results are presented in Figures \ref{fig2} and \ref{fig1}.
\begin{figure}[!ht]
\centering
\includegraphics[width=9cm]{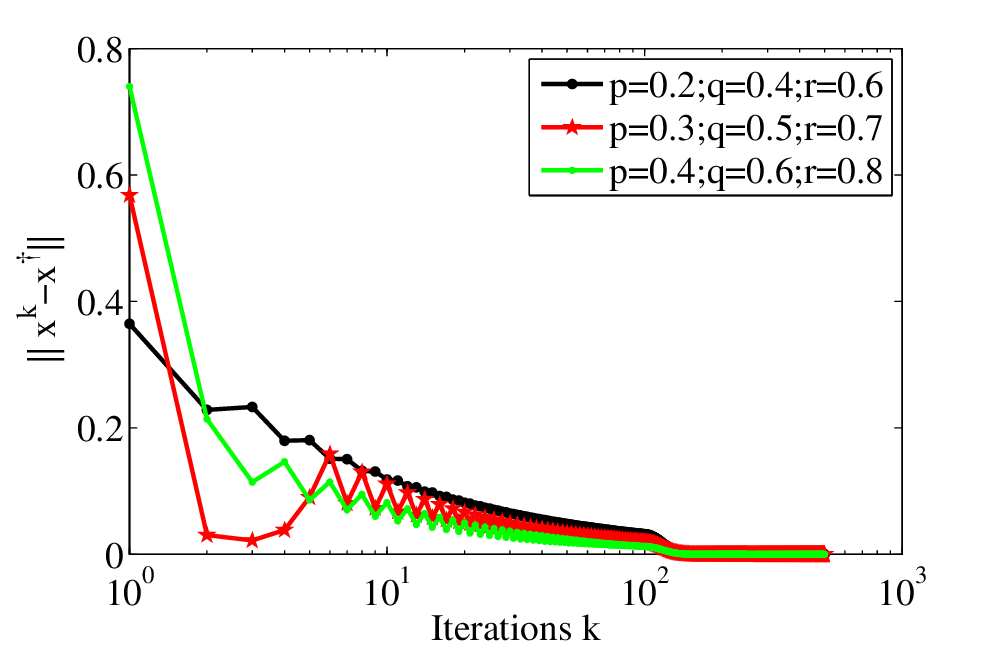}
\caption{Behaviors of the sequence $\{x^k\}$ generated by Algorithm \eqref{Discret} in Example  \ref{ex2}}\label{fig2}
\end{figure}
\begin{figure}[!ht]
\centering
\includegraphics[width=9cm]{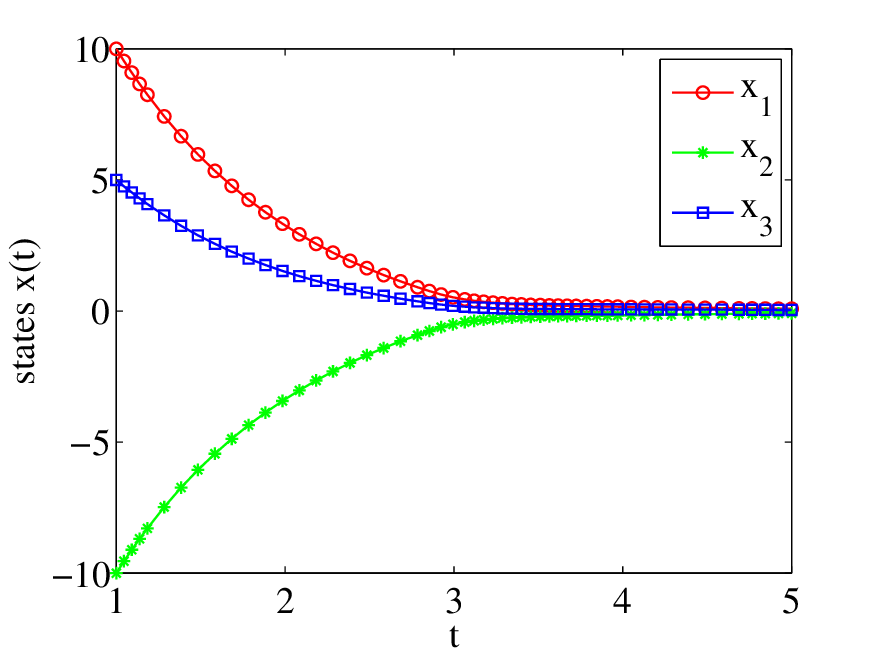}
\caption{Behaviors of the function $x(t)$ generated by Algorithm \eqref{RegNN} in Example  \ref{ex2}}\label{fig1}
\end{figure}
\end{example}
\begin{example}\label{example3}
    Let $C:=\{ x=(x_1,x_2,\ldots,x_n)^T \in \mathbb{R}^n: x_i\in [-1;1], i=1,\ldots n  \}$, $A:\mathbb{R}^n \to \mathbb{R}^n$, $A(x)=(x_1, 2x_2,\ldots,n x_n )$,  $F: \mathbb{R}^n \to \mathbb{R}^n$, $F(x)=Bx$, where $B=(2n)^{-1}(b_{ij})_{n\times n}$ and
    $$
    b_{ij}=\begin{cases}
        0\text{ if } i=j,\\
        i+j\text{ if } j>i,\\
        -j(i+j)/i\text{ if } j<i.
    \end{cases}
    $$
    It is easy seen that $A,F$ is Lipschitz continuous, $A$ is $1$-strongly monotone. Moreover, since $ \left\langle Ax,Fx \right\rangle =0  $ for all $x \in \mathbb{R}^n$, we have
    $$
    \left\langle Ax-Ay,Fx-Fy \right\rangle  =0\quad \forall x,y\in \mathbb{R}^n.
    $$
    Hence, the couple $(A,F)$ is monotone. Next, we show that Sol$(A; F ; C)=\{0\}$.
    Indeed, suppose that $x^*$ is a solution of the problem \ref{GVI}. Thus, $ \left\langle Ax^*,y-Fx^* \right\rangle \geq 0  $ for all $y\in C$. Since $ \left\langle Ax^*,Fx^* \right\rangle =0 $, hence $ \left\langle Ax^*,y\right\rangle \geq 0   $ for all $y\in C.$ If $Ax^* =0$, then $x^*=0$, otherwise, taking $y=-\frac 1 {\|Ax^*\|}Ax^* \in C$, we have $ - \left\langle Ax^*,Ax^* \right\rangle \geq 0  $, which implies $x^* =0.$ Thus, $x^* =0$ is the unique solution of the problem \ref{GVI}. We apply the dynamical system \eqref{RegNN} to solve this problem with $\alpha (t)=(t+1)^{-0.2} $, $\mu (t)=(t+1)^{0.4}.$
The starting point is randomly generated. We implement the algorithm with different $n$. The results are presented in Figure \ref{fig3}. As we can see, in all the cases, the function $x(t)$ converges to the unique solution $x^* =0$ of the problem.
\begin{figure}[!ht]
    \centering
   \includegraphics[width=\textwidth]{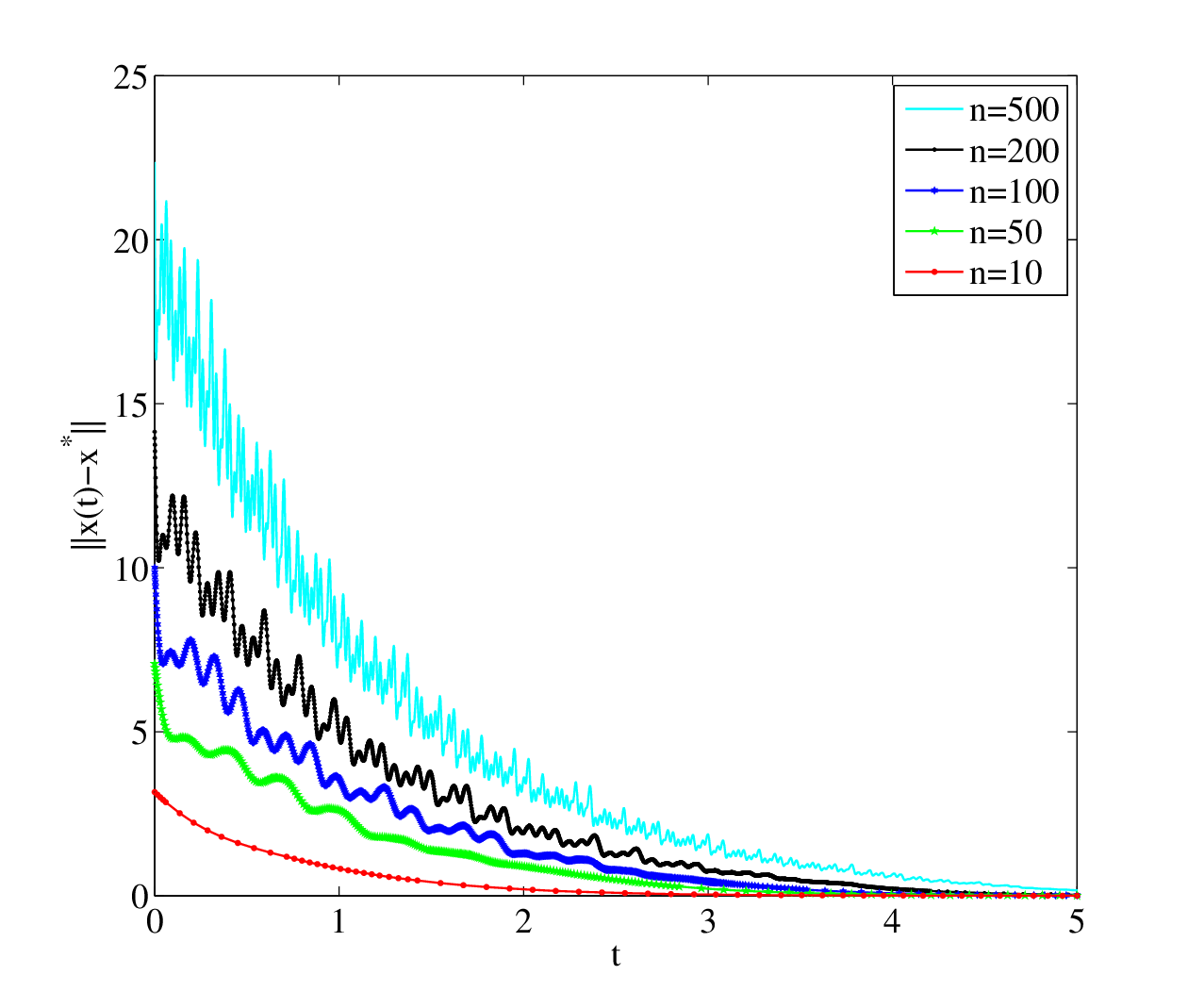}
    \caption{Behaviors of the function $x(t)$ generated by the dynamical system \eqref{RegNN} in Example \ref{example3}}\label{fig3}
\end{figure}
\end{example}
\section{Conclusions}	
In this paper, based on the regularization techniques, we introduce neural networks for solving  general variational inequalities involving monotone couples of operators. We prove that under some suitable conditions, the regularized neural network as well as its discretized version strongly converge to a certain solution of the GVI. We also give some numerical experiments, including infinite dimensional one, to certify the effectiveness of the proposed algorithm.
{\bf Open questions.} {\it Do Theorems \ref{Theorem_continuous_convergence} and \ref{theo2discret} still hold when $A$ is not necessary a bounded linear self-adjoint positive define operator ?. }

\end{document}